\documentclass[12pt]{amsart}
\usepackage{amsmath, amsfonts,amsthm,times,graphics}

\usepackage[active]{srcltx}

 \makeatletter
\renewcommand*\subjclass[2][2010]{%
  \def\@subjclass{#2}%
  \@ifundefined{subjclassname@#1}{%
    \ClassWarning{\@classname}{Unknown edition (#1) of Mathematics
      Subject Classification; using '2010'.}%
  }{%
    \@xp\let\@xp\subjclassname\csname subjclassname@#1\endcsname
  }%
}

 \makeatother
\newtheorem{theorem}{Theorem}[section]
\newtheorem{lemma}[theorem]{Lemma}

\newtheorem{conjecture}[theorem]{Conjecture}
\theoremstyle{definition}

\newtheorem{remark}[theorem]{Remark}

 \makeatletter
\renewcommand*\subjclass[2][2010]{%
  \def\@subjclass{#2}%
  \@ifundefined{subjclassname@#1}{%
    \ClassWarning{\@classname}{Unknown edition (#1) of Mathematics
      Subject Classification; using '1991'.}%
  }{%
    \@xp\let\@xp\subjclassname\csname subjclassname@#1\endcsname
  }%
}

 \makeatother

\begin{document}
\title{Two curious  inequalities involving different means of two arguments}

\author{Romeo Me\v strovi\' c}
\address{Maritime Faculty Kotor, University of Montenegro, Dobrota 36,
85330 Kotor, Montenegro} \email{romeo@ac.me}

\author{Miomir Andji\'{c}}

\address{ Faculty for Information Technology\\  
University ``Mediterranean''\\ Vaka Djurovi\'{c}a bb\\ Podgorica\\ Montenegro}

\email{miomir.andjic@unimediteran.net}

 
\begin{abstract}
For two positive real numbers $x$ and $y$
let $H$, $G$, $A$ and 
$Q$ be  the harmonic mean, the geometric mean, the arithmetic 
mean and the quadratic mean of $x$ and $y$, respectively. In this note, 
we prove that 
      \begin{equation*}
 A\cdot G\ge Q\cdot H,
    \end{equation*}
and that for each  integer $n$ 
    \begin{equation*}
A^n+G^n\le Q^n+H^n.
    \end{equation*}
We also  discuss and compare the first and the second  
above inequality for  $n=1$ with some known 
inequalities involving the mentioned classical means, the Seiffert mean $P$,
the logarithmic mean $L$ and the identric mean $I$ of two positive real  
numbers $x$ and $y$.  
   \end{abstract}  
  \maketitle

\section{The main result}

For two positive real numbers $x$ and $y$, let 
$H(x,y)=H$, $G(x,y)=G$, $A(x,y)=A$ and 
$Q(x,y)=Q$ be the harmonic mean, the geometric mean, the arithmetic mean and 
the quadratic mean (sometimes called the root mean square) 
of $x$ and $y$, respectively, i.e., 
  $$
H=\frac{2xy}{x+y},\,\, G=\sqrt{xy},\,\, 
A=\frac{x+y}{2},\,\, \mathrm{and}\,\,Q=\sqrt{\frac{x^2+y^2}{2}}.
  $$
Then by the particular case of the well known 
harmonic mean-geometric mean-arithmetic mean-quadratic mean inequality
($H-G-A-Q$ inequality), 
 $$
H\le G\le A\le  Q,
$$ 
with equality if and only if $x=y$.

Many sources have discussed one or more of the inequalities involving harmonic 
mean, geometric mean, arithmetic mean, and quadratic mean
(see e.g., \cite{al}, \cite {bmv} and  \cite{hlp}).  
In this note, under the above notations, we will prove the following result.
   \begin{theorem}\label{th} 
Let $x$ and $y$ be arbitrary positive real numbers, and let $n$ be any 
 integer.  Then
 \begin{equation}\label{eq2}
 A\cdot G\ge Q\cdot H,
    \end{equation}
and 
   \begin{equation}\label{eq1}
A^n+G^n\le Q^n+H^n,
    \end{equation}
        
The equality in \eqref{eq2} and \eqref{eq1}  holds if and only if $x=y$.
 \end{theorem}

\begin{remark} In particular, the inequality \eqref{eq1} implies that
  \begin{equation}\label{eq3}
A+G\le Q+H.
    \end{equation}
Notice that in \cite[(3.2) of Theorem 1]{sz}
(also see  \cite{sz2})  J. S\'{a}ndor
  proved that
for all $x>0$ and $y>0$
     \begin{equation}\label{eq4}
A+G\le 2P,
    \end{equation}
where $P=P(x,y)$ is the Seiffert mean of two positive real numbers
$x$ and $y$  defined by 
      \begin{equation}\label{eq5}
P=P(x,y)=\frac{x-y}{2\arcsin\frac{x-y}{x+y}}\,\,{\rm if\,\,}x\not= y,
\,\, {\rm and\,\,} P(x,x)=x.
    \end{equation}
The equality in \eqref{eq4} holds if and only if $x=y$. In view of the inequalities 
\eqref{eq3} and \eqref{eq4}, 
it can be of interest to compare 
the expressions $Q+H$ and $2P$. Our computational results suggest that 
the inequality \eqref{eq3} is stronger than  the inequality 
\eqref{eq4}, i.e., that it is true
the following conjecture.
 \end{remark}

\begin{conjecture}\label{con} 
Let $x$ and $y$ be arbitrary positive real numbers. Then 
  \begin{equation}\label{eq6}
Q+H\le 2P,
    \end{equation}
 where equality holds if and only if $x=y$.
 \end{conjecture}

 \begin{remark}
The logarithmic mean $L=L(x,y)$ and the identric  
mean  $I=I(x,y)$  of two positive real numbers
$x$ and $y$ are   defined by 
    \begin{equation}\label{eq7}
L=L(x,y)=\frac{x-y}{\ln x-\ln y}\,\,{\rm if\,\,}x\not= y,
\,\, {\rm and\,\,} L(x,x)=x;
    \end{equation}
    \begin{equation}\label{eq8}
I=I(x,y)=\frac{1}{e}\left(\frac{y^y}{x^x}\right)^{\frac{1}{y-x}}
\,\,{\rm if\,\,}x\not= y,\,\, {\rm and\,\,} I(x,x)=x.
    \end{equation}
In \cite{al1} H. Alzer proved that for all $x>0$ and $y>0$  we have
  \begin{equation}\label{eq9}
\sqrt{A\cdot G}\le \sqrt{L\cdot I}\le \frac{L+I}{2}
\le \frac{G+A}{2}, 
   \end{equation}
where the equality in each of these inequalities holds if and only if $x=y$.
Notice that in view of inequalities \eqref{eq2}, \eqref{eq1} and 
\eqref{eq3}, the chain of inequalities 
given by \eqref{eq9} may be extended as 
  \begin{equation}\label{eq10}\begin{split}
\sqrt{Q\cdot H}& \le \sqrt{A\cdot G}\le \sqrt{L\cdot I}\\
&\le  \frac{L+I}{2}\le \frac{G+A}{2}\le \frac{Q+H}{2}. 
   \end{split}\end{equation}
Moreover, under Conjecture \ref{con} and the known fact that $P\le I$
(see \cite{se}), the  chain of inequalities
\eqref{eq10} may be extended on the right hand side as
 \begin{equation*}
\frac{Q+H}{2}\le P\le I. 
 \end{equation*}
   \end{remark}

\begin{remark} Since the inequality \eqref{eq1}  is satisfied 
for each integer $n$, it may be of interest to  answer  the
following question: For which  real numbers $n$ the inequality 
\eqref{eq1} holds? Our computational and related graphical results lead to the following 
conjecture.
    \end{remark}

\begin{conjecture}\label{con1.5}
The inequality \eqref{eq1} holds 
for all negative real numbers $n$ and for all positive real numbers 
$n$ greater or equal than $1/2$. Moreover, 
none of the inequality \eqref{eq1} or its converse  inequality holds true 
for each real number $n$ in the interval $(0,1/2)$. 
 \end{conjecture}  
 
\section{Proof of Theorem \ref{th}}

For the proof of the inequality \eqref{eq1} of Theorem \ref{th} we will need 
the following lemma.

\begin{lemma}\label{le} 
Let $a$, $b$, $c$ and $d$ be positive  real numbers such that $a+b\le c+d$
and $ab\ge cd$. Then for each integer $n$ 
   \begin{equation}\label{eq11}
a^n+b^n\le c^n+d^n.
     \end{equation}
\end{lemma}
 \begin{proof}
First we will prove the inequality \eqref{eq11} for nonnegative integers $n$.
We proceed by induction on $n\ge 0$. 
Obviously, the inequality \eqref{eq11} is satisfied for $n=0$.
Suppose that the inequality \eqref{eq11} holds for all nonnegative integers 
$\le n$. Then using this hypothesis and the assumption $ab\ge cd$, we obtain.  
  \begin{equation*}\begin{split}
a^{n+1}+b^{n+1}&=(a^n+b^n)(a+b)-ab(a^{n-1}+b^{n-1})\\
&\le (c^n+d^n)(c+d)-cd(c^{n-1}+d^{n-1})\\ 
& = c^{n+1}+d^{n+1}.
     \end{split}\end{equation*}
Hence, $a^{n+1}+b^{n+1}\le c^{n+1}+d^{n+1}$, which completes the 
induction proof.

Now suppose that $n$ is a negative integer. Then applying the inequality 
\eqref{eq11} with $-n>0$ instead of $n$ and the assumption that $ab\ge cd$, we find that
  \begin{equation*}
a^n+b^n =\frac{a^{-n}+b^{-n}}{(ab)^{-n}}\le 
\frac{c^{-n}+d^{-n}}{(cd)^{-n}}= c^{-n}+d^{-n}.
     \end{equation*}
 Hence, the inequality \eqref{eq11} holds for each integer $n$.   
\end{proof}

\begin{proof}[Proof of Theorem \ref{th}]
In order to prove the inequality \eqref{eq2}, notice that by the identity 
$(x-y)^4=(x+y)^4-8xy(x^2+y^2)$ we obtain 
       \begin{equation}\label{eq12}
(x+y)^2\ge 2\sqrt{2xy(x^2+y^2)}.
       \end{equation}
By using the inequality \eqref{eq12}, we get 
     \begin{equation*}
    \frac{A\cdot G}{Q\cdot H}=\frac{(x+y)^2}{2\sqrt{2xy(x^2+y^2)}}\ge 1,
       \end{equation*}
which implies the inequality \eqref{eq1}.

It remains to prove the inequality \eqref{eq1}. Notice that by Lemma \ref{le}
(with $a=A$, $b=G$, $c=Q$ and $d=H$)  and 
the inequlaity \eqref{eq2}, it suffices to prove the inequality 
\eqref{eq1} for $n=1$.    

First observe that 
   \begin{equation}\label{eq13}
A-H=\frac{x+y}{2}-\frac{2xy}{x+y}= \frac{(x-y)^2}{2(x+y)}.
    \end{equation}
By using $A-Q$ inequality, we have
     \begin{equation}\label{eq14}
\frac{\sqrt{2(x^2+y^2)}+\sqrt{4xy}}{2}\le 
\sqrt{\frac{2(x^2+y^2)+4xy}{2}}=x+y.
  \end{equation}
Then applying    the inequality \eqref{eq14} and the identity \eqref{eq13}, 
we obtain
   \begin{equation*}\begin{split}
Q-G&=\sqrt{\frac{x^2+y^2}{2}}-\sqrt{xy}=
\frac{\frac{x^2+y^2}{2}-xy}{\sqrt{\frac{x^2+y^2}{2}}+\sqrt{xy}}\\
&=\frac{(x-y)^2}{\sqrt{2(x^2+y^2)}+\sqrt{4xy}}\ge 
\frac{(x-y)^2}{2(x+y)}\\
&= A-H,
   \end{split}\end{equation*}
which implies the inequality \eqref{eq1} for $n=1$. 

From the above proofs it follows that the equality in 
\eqref{eq2} and \eqref{eq1} holds if and only if  $x=y$.
This completes proof of Theorem \ref{th}.
\end{proof}

\end{document}